\documentclass[reqno,12pt]{amsart}

\usepackage{a4wide}
\usepackage{mathtools}
\usepackage{amsthm}
\usepackage{amssymb}
\usepackage[colorlinks,citecolor=green,linkcolor=red]{hyperref}
\usepackage{color}
\mathtoolsset{showonlyrefs}
\usepackage{csquotes}
\usepackage{graphicx}
\usepackage{psfrag}
\usepackage{epsfig,color}
\usepackage{enumitem}

\numberwithin{equation}{section}
\newtheorem{theorem}{Theorem}[section]
\newtheorem{lemma}[theorem]{Lemma}

\theoremstyle{definition}

\newtheorem{example}[theorem]{Example}
\theoremstyle{remark}

\renewcommand{\d}{{\mathrm d}}
\newcommand{\R}{\mathbb{R}}

\newcommand{\N}{\mathbb{N}}

\newcommand{\diam}{\mathrm{diam}}

\newcommand{\dist}{\mathrm{dist}}

\begin{document}

\title[Dimension of boundaries of extension domains]
{Dimension estimates for the boundary of\\ planar Sobolev extension domains}
\author{Danka Lu\v{c}i\'c}
\author{Tapio Rajala}
\author{Jyrki Takanen}

\address{University of Jyvaskyla\\
         Department of Mathematics and Statistics \\
         P.O. Box 35 (MaD) \\
         FI-40014 University of Jyvaskyla \\
         Finland}
\email{danka.d.lucic@jyu.fi}
\email{tapio.m.rajala@jyu.fi}
\email{jyrki.j.takanen@jyu.fi}

\thanks{All authors partially supported by the Academy of Finland, 
 project 314789.}
\date{\today}


\begin{abstract}
 We prove an asymptotically sharp dimension upper-bound for the boundary of bounded simply-connected planar Sobolev $W^{1,p}$-extension domains via the weak mean porosity of the boundary. The sharpness of our estimate is shown by examples.
\end{abstract}

\maketitle
\tableofcontents
\section{Introduction}

A set is porous if it has holes arbitrarily close to any point, and those holes have diameter comparable to the distance to the point. It is easy to see that porous sets in $\R^d$ have zero Lebesgue measure. 
If the porosity of the set $A\subset \R^d$ is stronger, in the sense that 
\[
\textrm{por}(A) \coloneqq \inf_{x\in A}\liminf_{r \searrow 0}\textrm{por}(A,x,r) > 0,
\]
where we denote the maximal size of a hole of the set $A \subset \R^d$ at $x \in \R^d$ and of scale $r > 0$ by
\[
\textrm{por}(A,x,r) \coloneqq \sup\{\alpha \ge 0\,:\, \text{there exists }y\in \R^d\text{ such that } B(y,\alpha r) \subset B(x,r)\setminus A\},
\]
then the Hausdorff dimension of $A$ is strictly less than $d$.  It was shown by Mattila  \cite{M1988} that as $\textrm{por}(A)$ gets closer to its maximal value $\frac12$, the dimension upper-bound for $A$ goes to $d - 1$. The sharp asymptotic behaviour when $\textrm{por}(A)\to \frac12$ was then established by Salli in  \cite{S1991}.
Later, several variants of porosity have been considered. For example, in a variant of porosity called $k$-porosity, one looks at $k$ holes in orthogonal directions, instead of just one, see \cite{JJKS2005,KS2011}. For it, the dimension upper-bound approaches $d-k$ as the porosity goes to its maximal value.

In the present paper we are  interested in the asymptotic behaviour of the dimension upper-bound when $\textrm{por}(A)\to 0$. In this case, for the usual porosity defined above we have the sharp upper-bound
\[
\dim_\mathcal{H}(A) \le d - c\,\textrm{por}(A)^d,
\]
for some constant $c$ depending on the dimension, see for instance \cite{MV1987}. However, sometimes we are in a setting where the porosity condition is not satisfied in the exact form as stated above, but almost.
One such instance is the study of growth conditions on the hyperbolic metric, which imply the existence of holes only in a portion of the scales, but not all scales.
Motivated by this, Koskela and Rohde introduced a version of porosity called \emph{mean porosity} and proved a sharp dimension upper bound for mean porous sets \cite{Koskela_Rohde_97} (see also the estimates by Beliaev and Smirnov \cite{BS2002} that deal also with a generalization of Salli's result). 

Our aim in this paper is to show sharp dimension bounds for boundaries of Sobolev extension domains.  For obtaining these, even the mean porosity of Koskela and Rohde is not flexible enough, because we might have many holes in a more sparse set of scales.
Therefore, we use a variant of mean porosity introduced by Nieminen in \cite{N2006}, called \emph{weak mean porosity} (see Section \ref{sec:poro} for the definition).
\medskip

Recall that
a domain $\Omega \subset \R^d$ is called a Sobolev $W^{1,p}$-extension domain, if there exists a constant $C \in (1,\infty)$ so that for every $f \in W^{1,p}(\Omega)$
there exists $F \in W^{1,p}(\R^d)$ so that $F|_\Omega = f$ and $\|F \|_{W^{1,p}(\R^d)} \le C\|f \|_{W^{1,p}(\Omega)}$. When $p>1$, the operator $f \mapsto F$ can always be assumed to be linear \cite{HKT2008}. In \cite{sh2010} and \cite{KRZ15}, bounded simply-connected Sobolev extension domains  $\Omega \subset \R^2$ were characterized by a curve condition, which for the range $1 < p < 2$ is the following: There exists a constant $C > 1$ such that for every $z_1,z_2 \in \R^2 \setminus {\Omega}$ there exists a curve  $\gamma \subset \R^2 \setminus {\Omega}$ connecting $z_1$ and $z_2$ and satisfying
 \begin{equation}\label{eq:curve_condition_intro}
  \int_{\gamma} \dist(z,\partial \Omega)^{1-p}\,{\rm d}s(z) \leq C\|z_1-z_2\|^{2-p}.
 \end{equation}
 We give an upper bound on the Hausdorff dimension $\dim_{\mathcal H}$ of the boundary of $\Omega$ in terms of the constant $C$ in \eqref{eq:curve_condition_intro}. This is done by showing the weak mean porosity of the boundary in Theorem \ref{thm:main_thm} and by combining it with the dimension estimate proven by Nieminen (Theorem \ref{thm:dim_estim}). The result we obtain is the following.
 \begin{theorem}\label{thm:main_intro}
 There exists a universal constant $M > 0$ such that for every  bounded simply-connected domain $\Omega \subset \R^2$ satisfying the curve condition \eqref{eq:curve_condition_intro} with some
 $ C \in (1,\infty)$  the following holds:
  \begin{equation}\label{eq:dim_bound_intro}
 \dim_{\mathcal H}(\partial\Omega) \leq 2-\frac{M}{C}.
 \end{equation}
 \end{theorem}

In Section \ref{sec:examples}, we show that Theorem \ref{thm:main_intro} is sharp in the sense that there exists another constant $M' > 0$ so that for every $p\in(1,2)$ and $C \in (M'/(2-p),\infty)$ there exists a Jordan domain $\Omega_C \subset \R^2$ satisfying \eqref{eq:curve_condition_intro} with
 \[
   \dim_{\mathcal H}(\partial\Omega_C) \ge 2-\frac{M'}{(2-p)C}.
 \]
 Notice, however, the factor $\frac{1}{2-p}$ difference between Theorem \ref{thm:main_intro} and the examples.
 The curve condition \eqref{eq:curve_condition_intro} implies that $\mathbb R^2 \setminus \Omega$ is quasi-convex. Consequently, the domain $\Omega$ is a $J$-John domain \cite{NV1991}, meaning that there exists a constant $J > 0$ and a point $x_0 \in \Omega$  so that for every $x \in \Omega$ there exists a unit speed curve $\gamma\colon [0,\ell(\gamma)]\to \Omega$ such that $\gamma(0) = x$, $\gamma(\ell(\gamma)) = x_0$, and 
 \begin{equation}\label{eq:John}
 \dist(\gamma(t), \partial \Omega) \ge Jt \qquad \text{for all }t \in [0,\ell(\gamma)].
 \end{equation}
 Koskela and Rohde showed that the boundary of a $J$-John domain $\Omega \subset \R^2$ has the dimension bound
 \begin{equation}\label{eq:John_dim}
   \dim_{\mathcal H}(\partial\Omega) \leq 2-cJ,
 \end{equation}
 for some constant $c>0$. In Section \ref{sec:examples} we show that the bound \eqref{eq:John_dim} is also sharp.

 In Section \ref{sec:examples} we also show that from the curve condition, via the John condition and  the mean porosity of Koskela and Rohde \cite{Koskela_Rohde_97}, it is not possible to get a better bound than
  \begin{equation}\label{eq:dim_bound_nonsharp}
  \dim_{\mathcal H}(\partial\Omega) \leq 2 - \frac{M}{((2-p)C)^{1/(2-p)}}.
 \end{equation}
  A reason why the John condition does not give the sharper bound is that using it we consider holes only in the domain (or its complement), whereas by going from the curve condition directly to weak mean porosity, we can use holes on both sides of the boundary.
  
\section{Preliminaries}\label{sec:preli}

Let us start by introducing some notation and preliminary results.
By a \emph{cube} in $\R^d$ we mean an open cube whose sides are 
parallel to the axes in $\R^d$. The side-length of a cube $Q\subset\R^d$ 
will be denoted by $\ell(Q)$. 
By a dyadic cube $Q$ we mean that it is of the form
 \[
 Q = (i_12^{-k}, (i_1+1)2^{-k}) \times (i_22^{-k}, (i_2+1)2^{-k}) \times \cdots \times (i_d2^{-k}, (i_d+1)2^{-k})
 \]
 for some $k, i_1, i_2, \dots, i_d \in \mathbb Z$. We denote the set of dyadic cubes in $\R^d$ by $\mathcal D_d$.
Given an open non-empty set $U\subset \R^d$ that is not the whole $\R^d$, we denote by $\mathcal{W}_U$ the 
\emph{Whitney decomposition} of $U$, 
defined as
\[
 \mathcal{W}_U = \{Q \in \widetilde{\mathcal{W}}_U \,:\, \text{if }Q' \in \widetilde{\mathcal{W}}_U \text{ with }Q'\cap Q \ne \emptyset, \text{ then }Q' \subset Q\},
\]
where
\[
\widetilde{\mathcal{W}}_U = \{Q \in \mathcal D_d \,:\, \text{if }Q' \in \mathcal D_d \text{ with }\overline{Q} \cap \overline{Q}'\ne \emptyset \text{ and }\ell(Q) = \ell(Q') \text{ then }Q' \subset U\}.
\]
See Figure \ref{fig:cubes} for an illustration of the Whitney decomposition.
 \begin{figure}
    \includegraphics[width=0.6\columnwidth]{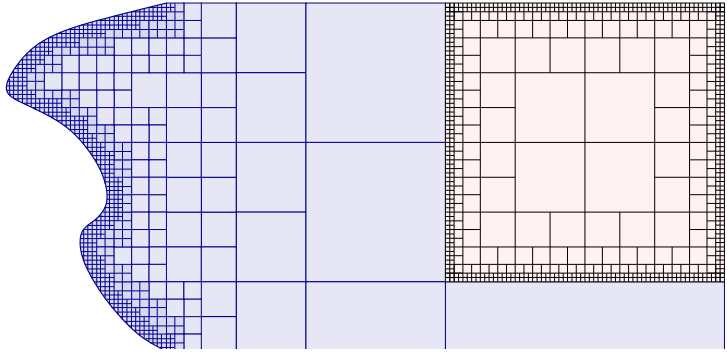}
    \caption{In our proof, we will use a double dyadic decomposition similar to the one used in \cite{Koskela_Rohde_97}. A domain is first decomposed into its Whitney cubes. Then each Whitney cube is decomposed into its own Whitney cubes, as illustrated here only for the largest cube in the first decomposition.}
    \label{fig:cubes}
\end{figure}
It readily follows that $\mathcal W_U$ is a collection of pairwise disjoint dyadic cubes $Q$ so that
$U=\bigcup_{Q\in \mathcal W_U}\overline{Q}$.
Moreover, the following condition is satisfied by each $Q\in \mathcal W_U$:
\begin{equation}\label{eq:Whitney_property}
\ell(Q)\leq {\rm dist}(Q,\partial U)\leq 4\,{\rm diam}(Q)
=4\sqrt{d}\,\ell(Q).
\end{equation}
Moreover, if $Q,Q' \in \mathcal{W}_U$ with $\overline{Q} \cap \overline{Q}' \ne \emptyset$, then
\begin{equation}\label{eq:factor2}
 \frac12 \le \frac{\ell(Q)}{\ell(Q')} \le 2.
\end{equation}

In the specific case where we take the Whitney decomposition of a dyadic cube $Q \in \mathcal D_d$, we have
\begin{equation}\label{eq:W_Qdef}
 \mathcal W_Q = \left\{Q' \subset Q \,:\, Q' \text { dyadic cube with }\ell(Q') = \dist(Q',\partial Q) \right\}.
\end{equation}
See again Figure \ref{fig:cubes} for an illustration.
It is then easy to check that 
\begin{equation}\label{eq:W_Qbound}
\#\left\{Q'\in \mathcal W_Q:\, 
\ell(Q')=2^{-j}\ell(Q)\right\}\geq 2^{(j-1)(d-1)}\quad \text{ holds for every }j \ge 2.
\end{equation}

%
%
%
%
\medskip
Given any ball $B\subset \R^d$ and any $r>0$, we denote by $rB$
the ball having the same center as $B$ and the radius $r$ times that of $B$.
The ball of radius $r>0$, centered in $x\in \R^d$ is denoted by $B(x,r)$,
while by $B(E,r)$ we denote the $r$-neighbourhood 
of a given set $E\subset\R^d$.
\smallskip

Recall that the \emph{Hausdorff dimension} of a set $E\subset\R^d$ is defined by 
\begin{equation}
{\rm dim}_{\mathcal H}(E)\coloneqq \inf \{s>0: \mathcal H^s(E)=0\}
=\sup\{s>0: \mathcal H^s(E)=+\infty\},
\end{equation}
where $\mathcal H^s$ stands for $s$-dimensional Hausdorff measure in $\R^d$.
%
%

\subsection{Weakly mean porous sets}\label{sec:poro}

In the present subsection, we recall the concept of
weak mean porosity introduced in \cite{N2006}. The weak mean porosity is a variant of mean porosity introduced in \cite{Koskela_Rohde_97}.

Let $E \subset\R^d$ be a compact set. Let $\alpha \colon ]0,1[ \to ]0,1[$ be a continuous function such that $\alpha(t)/t$ is increasing in $t$, and let $\lambda \colon \mathbb Z^+ \to \mathbb R$ be a function. Let $\mathcal D$ be a disjointed collection of open cubes in $\R^d\setminus E$. Define
\[
 \chi_k^\mathcal{D}(x) = \begin{cases}
                          1, &\text{if there exist at least }\lambda(k)\text{ cubes }Q \in \mathcal{D}\text{ with }Q \subset A_k(x) \text{ and }\ell(Q) \ge \alpha(2^{-k}),\\
                          0, &\text{otherwise},
                         \end{cases}
\]
where $A_k(x)\coloneqq B(x,2^{-k})\setminus B(x,2^{-k-1})$. Let
\[
 S_j^\mathcal{D}(x) = \sum_{k=1}^j \chi_k^\mathcal{D}(x).
\]
We say that $E$ is \emph{weakly mean porous with parameters $(\alpha,\lambda)$}, if there exists a collection $\mathcal D$ and $j_0 \in \mathbb Z^+$ such that
\[
 \frac{S_j^\mathcal{D}(x)}{j} > \frac12
\]
for all $x \in E$ and for all $j \ge j_0$.

We will apply weak mean porosity in the case 
\begin{equation}\label{eq:lambdaalpha}
 \lambda(k) = c \varepsilon^{-1}\qquad \text{and} \qquad\alpha(t) = \varepsilon t, 
\end{equation}
for some $\varepsilon \in ]0,1[$ and a fixed constant $c > 0$.
In this case, we have the following dimension estimate as a direct corollary of \cite[Theorem 3.3]{N2006}.
\begin{theorem}\label{thm:dim_estim}
 There exists a constant $C(d,c)> 0$ such that any weakly mean porous set $E \subset \R^d$ with parameters $(\alpha,\lambda)$ defined in \eqref{eq:lambdaalpha} satisfies
 \[
  \dim_{\mathcal{H}}(E) \le d - C(d,c)\varepsilon^{d-1}.
 \]
\end{theorem}

\section{Weak mean porosity of the boundary of Sobolev extension domains}
\label{sec:dimension_estimate}
In this section we will show that
the boundary of a planar bounded simply-connected $W^{1,p}$-extension domain 
(with $1<p<2$) is weakly mean porous with the parameters depending on the constant $C$ appearing in the \emph{curve condition}
\eqref{eq:curve_condition} that characterises $W^{1,p}$-extension domains 
(cf.\ Theorem \ref{thm:curve_condition} below).
\medskip

The following result has been proven in \cite{KRZ15}:
\begin{theorem}\label{thm:curve_condition}
Let $1<p<2$ and let $\Omega\subset \R^2$ be a bounded simply-connected domain.
Then, $\Omega$ is a $W^{1,p}$-extension domain if and only if there exists 
a constant $C=C(\Omega, p)>0$ such that every 
$z_1,z_2\in \R^2\setminus \Omega$ can be joined by a rectifiable curve 
$\gamma\in \R^2\setminus \Omega$ satisfying 
\begin{equation}\label{eq:curve_condition}
\int_{\gamma} \dist(z,\partial \Omega)^{1-p}\,{\rm d}s(z) \leq C\|z_1-z_2\|^{2-p}.
\end{equation}
\end{theorem}
Now we are ready to state our main result.

\begin{theorem}
\label{thm:main_thm}
There exist universal constants $C',C''>0$ so that the following holds.
Let $1<p<2$ and let $\Omega\subset \R^2$ be a bounded simply-connected 
$W^{1,p}$-extension domain. Let $C$ be the constant from the curve condition \eqref{eq:curve_condition}.
Then, $\partial \Omega$ is weakly mean porous with parameters $(\alpha,\lambda)$, where $\lambda(k) = C'C$ and $\alpha(t) = \frac{C''}{C}t$.
\end{theorem}
In the proof of Theorem \ref{thm:main_thm}, we use the following result to relate the length of the curve $\gamma$ in \eqref{eq:curve_condition} to the diameter of cubes it intersects.

\begin{lemma}\label{lma:shortcuts}
 Let $1<p<2$, let $\Omega\subset \R^2$ be a bounded simply-connected 
$W^{1,p}$-extension domain and let $z_1,z_2 \in \R^2\setminus \Omega$. Then there exists a curve $\gamma$ connecting $z_1$ and $z_2$ in $\R^2 \setminus \Omega$ that minimizes
  \begin{equation}\label{eq:curve_min}
\int_{\gamma} \dist(z,\partial \Omega)^{1-p}\,{\rm d}s(z) 
\end{equation}
 and satisfies 
 \[
  \mathcal{H}^1(\gamma \cap \overline{Q}) \le 10\, \ell(Q)
 \]
 for every $Q \in \mathcal W_{\R^2 \setminus \overline{\Omega}}$.
\end{lemma}
\begin{proof}
 The existence of a minimizer for \eqref{eq:curve_min} is standard and has been established in the proof of \cite[Lemma 2.17]{KRZ15}.
 Let $Q \in \mathcal W_{\R^2 \setminus \overline{\Omega}}$. Define
  $t_1 = \min\{t \,:\, \gamma(t) \in \overline{Q}\}$
 and $t_2 = \max\{t \,:\, \gamma(t) \in \overline{Q}\}$. 
 Then, by \eqref{eq:Whitney_property} and the minimality of $\gamma$,
 \begin{align*}
 \mathcal{H}^1(\gamma\cap \overline{Q})\left(5\sqrt{2}\ell(Q)\right)^{1-p} 
 & \le  \mathcal{H}^1(\gamma\cap \overline{Q})\left(\dist(Q,\partial\Omega)+\diam(Q)\right)^{1-p}\\
 &\le \int_{\gamma\cap \overline{Q}} \dist(z,\partial \Omega)^{1-p}\,{\rm d}s(z) \\
 &\le \int_{[\gamma(t_1),\gamma(t_2)]}\dist(z,\partial\Omega)^{1-p}\,{\rm d}s(z)\\
 & \le \diam(Q)\dist(Q,\partial\Omega)^{1-p} \le \sqrt{2}\ell(Q)^{2-p}.
 \end{align*}
 Thus, the claim holds.
\end{proof}

\begin{proof}[Proof of Theorem \ref{thm:main_thm}]
Without loss of generality, we may assume that 
 $C \ge 1$. 
 Let $\varepsilon \coloneqq 2^{-m} \in (2^{-15}/C, 2^{-14}/C]$ with $m \in \mathbb Z$.
We start by constructing the collection $\mathcal D$ of cubes in $\mathbb R^2 \setminus \partial \Omega$.
We decompose every $Q\in \mathcal W_{\R^2\setminus \partial\Omega}$ 
into $\mathcal W_Q$ and enumerate
 $\mathcal W_Q = \{Q_i(Q)\}_{i\in \N}$. 
We will show that the family 
\[
\mathcal D \coloneqq \big\{Q_i(Q):\, i\in \N,\,Q\in \mathcal W_{\R^2\setminus \partial\Omega} \big\}
\]
gives the claimed weak mean porosity of $\partial\Omega$ with the functions $\lambda(k) = \varepsilon^{-1} 2^{-10}$ and $\alpha(t) = \varepsilon t$.

 Let $k_0$ be the smallest positive integer for which  $2^{-k_0}< \diam(\Omega)$.
 It suffices to show that $\chi_k^\mathcal{D}(x) = 1$ for all $k \ge k_0$ and 
 $x\in \partial \Omega$.
 Let us fix $k\in \N$, with $k\ge k_0$, and $x\in \partial \Omega$. 
\medskip

{\color{blue}\textsc{Case 1:}}
First, let us suppose that the following condition holds true:
\begin{equation}\label{eq:condition_1}
\text{For every } r\in \left[\frac{2}{3}\,2^{-k},\frac{5}{6}\,2^{-k}\right] 
\text{there exists }y\in \partial B(x,r) \text{ so that }
B(y,\varepsilon 2^{-k+5})\cap \partial \Omega=\emptyset.
\end{equation}
Consider the set of radii 
\[
R \coloneqq \left\{r \,:\, r = \frac232^{-k}+ \varepsilon2^{-k+6}i  \le \frac{5}{6}\,2^{-k}, i \in \N \right\}.
\]
  For each $r \in R$ we select a point $y_r\in \partial B(x,r)$ so that  $B(y_r,\varepsilon 2^{-k+5})\cap \partial \Omega=\emptyset$, as given by \eqref{eq:condition_1}. 
Now, given any $r \in R$, the set $B(y_r,\varepsilon 2^{-k+5}) \subset \R^2 \setminus\partial\Omega$ contains a dyadic square $Q$ of sidelength $\varepsilon 2^{-k+2}$ with distance at least $\varepsilon 2^{-k+2}$ to $\partial\Omega$. Thus, $\partial B(x,r) \cap Q \ne \emptyset$ for some $Q \in \mathcal W_{\R^2\setminus \partial\Omega}$ with $\ell(Q) \ge \varepsilon 2^{-k+2}$.
Since $x \in \partial\Omega$, $\diam(\Omega) > 2^{-k}$ and $\Omega$ is bounded and simply-connected, we have
\[
 \partial B(x,r)\cap \partial \Omega\ne\emptyset,
\]
and so also arbitrarily small cubes in $\mathcal W_{\R^2\setminus \partial\Omega}$ intersect $\partial B(x,r)$.
Consequently, taking into account \eqref{eq:factor2} there exists $Q_r\in \mathcal W_{\R^2\setminus \partial\Omega}$ with $\ell(Q_r) = \varepsilon 2^{-k+2}$ and
\[
 \partial B(x,r)\cap Q_r\ne\emptyset.
\]
By the bound \eqref{eq:W_Qbound}, there exists $Q_r' \in \mathcal W_{Q_r} \subset\mathcal D$ with $\ell(Q_r') =  \varepsilon 2^{-k}$.
Then the collection of cubes $\{Q_r' \,:\, r \in R\} \subset \mathcal D$ is disjointed. A simple calculation shows that we have $\# R \ge 2^{-9}/\varepsilon$.
Thus, $\chi_k^\mathcal{D}(x) = 1$.


\medskip
{\color{blue}\textsc{Case 2:}}
If the condition \eqref{eq:condition_1} is violated, we argue as follows:
Pick $r\in \big(\frac{2}{3}\,2^{-k},\frac{5}{6}\,2^{-k}\big)$ such that 
for every $y\in \partial B(x,r)$ it holds that 
$B(y,\varepsilon 2^{-k+5})\cap \partial \Omega\neq\emptyset$.
Let $\{y_i\}_{i=1}^m$ be a maximal $\varepsilon2^{-k+5}$-separated net of points in $\partial B(x,r)$ enumerated in a clockwise order around $x$.
Since $B(y_i,\varepsilon 2^{-k+5})\cap \partial \Omega\neq\emptyset$, we can select, for each $i$ a point $w_i \in B(y_i,\varepsilon2^{-k+5}) \setminus \Omega$. 
Let us denote $w_{m+1} = w_1$. We claim that for some $i\in \{1, \dots, m\}$ 
\begin{equation}\label{eq:exitclaim}
\text{any curve connecting }w_i\text{ to }w_{i+1}\text{ in }\R^2 \setminus \Omega \text{ must exit }B(w_i,2^{-k-3}).
\end{equation}
Suppose this is not the case. Then we can connect $w_i$ to $w_{i+1}$ by a curve $\sigma_i$ in $B(w_i,2^{-k-3}) \setminus \Omega$.
The concatenation $\sigma$ of $\sigma_1, \dots, \sigma_m$ is then contained in the annulus 
\[
 B(x,r+2^{-k-3}) \setminus B(x,r-2^{-k-3}) \subset  B(x, 2^{-k}) \setminus B(x,2^{-k-1})
\]
and has winding number $-1$ around $x$. However, since $x\in \partial \Omega$ and $\Omega \setminus B(x, 2^{-k}) \ne \emptyset$, the curve $\sigma$ then disconnects $\Omega$, which is impossible. Thus, we have the existence of $i$ for which \eqref{eq:exitclaim} holds.



Let $\gamma\colon[0,1]\to \R^2\setminus \Omega$ be a curve connecting $z_1 \coloneqq w_i$ and $z_2 \coloneqq w_{i+1}$ which 
minimizes \eqref{eq:curve_min}. Call 
$A\coloneqq \{z\in \gamma:\,{\rm dist}(z,\partial \Omega)> 5\sqrt{2}\varepsilon 2^{-k+2}\}$ and note that \eqref{eq:curve_condition} yields
\begin{align*}
( 5\sqrt{2}\varepsilon 2^{-k+2})^{1-p}\,\mathcal H^1(\gamma\setminus A)
& \leq \int_{\gamma\setminus A}{\rm dist}(z,\partial \Omega)^{1-p}\,{\rm d}s(z) \\
& \leq \int_{\gamma}{\rm dist}(z,\partial \Omega)^{1-p}\,{\rm d}s(z)
\leq C (\varepsilon 2^{-k+7})^{2-p}.
\end{align*}
Consequently, by the choice of $\varepsilon$, we have that 
\[
\mathcal H^1(\gamma\setminus A)\leq  2^{5(2-p)+2}(5\sqrt{2})^{p-1} \varepsilon C  2^{-k} \le 2^{10} \varepsilon C  2^{-k} \le 2^{-k-4} 
\]
 and hence
\begin{equation}\label{eq:H1_of_A}
\mathcal H^1(A\cap B(w_i,2^{-k-3}))=\mathcal H^1(\gamma \cap B(w_i,2^{-k-3}))-\mathcal H^1(\gamma\setminus A)
\geq 2^{-k-3}-2^{-k-4}
\ge 2^{-k-4}.
\end{equation}

Now, notice that by the choice of the radius $r$, the point $w_i$ and the factor $\varepsilon$, we get
\begin{equation}\label{eq:spaceinannulus}
 \dist(\R^2\setminus A_k(x), B(w_i,2^{-k-3})) \ge \frac162^{-k} - \varepsilon2^{-k+5}-2^{-k-3} \ge 2^{-k-6}.
\end{equation}
Write 
\[
 \mathcal Q \coloneqq \{Q \in \mathcal{W}_{\R^2 \setminus \overline{\Omega}} \,:\, \ell(Q) \ge \varepsilon 2^{-k+2} \text{ and } Q \cap B(w_i,2^{-k-3}) \ne \emptyset\}.
\]
Suppose first that there exists $Q \in \mathcal Q$ with $\ell(Q) \ge 2^{-k-7}$. Then, by the definition of the decomposition $\mathcal W_{Q}$ and by \eqref{eq:spaceinannulus} a square $Q' \in \mathcal W_{Q}$ with $\ell(Q') = \varepsilon 2^{-k}$ that is closest to $w_i$ satisfies
\begin{align*}
\dist(\R^2\setminus A_k(x), Q') & \ge \dist(\R^2\setminus A_k(x), B(w_i,2^{-k-3})) - \sqrt{2}\,\dist(Q',\partial Q) - \diam(Q')\\
& \ge 2^{-k-6} - \sqrt{2}\ell(Q') - \sqrt{2}\ell(Q') \ge 2^{-k-6} - \varepsilon2^{-k+2}.
\end{align*}
Therefore, by counting the consecutive squares of side-length $\varepsilon 2^{-k}$ in $\mathcal W_{Q}$ starting from this square, we obtain the estimate
\[
 \# \left\{Q' \in \mathcal D\,:\, Q' \in \mathcal W_{Q}, Q' \subset A_k(x)\text{ and }\ell(Q') = \varepsilon 2^{-k}\right\} \ge \frac{2^{-k-7}}{\varepsilon 2^{-k}} \ge \frac{2^{-7}}{\varepsilon} 
\]
and thus, $\chi_k^\mathcal{D}(x) = 1$.

Suppose then that for all $Q \in \mathcal Q$ we have $\ell(Q) \le 2^{-k-7}$.
Then, by \eqref{eq:spaceinannulus} for all $Q \in \mathcal Q$ we have $Q \subset A_k(x)$.
Notice that by \eqref{eq:Whitney_property},
$A$ is contained in the closure of the union of Whitney cubes $Q \in \mathcal W_{\R^2 \setminus \overline{\Omega}}$
with $\ell(Q) \ge \varepsilon 2^{-k+2}$ and that $\mathcal{H}^{1}$-almost every point in $\R^2$ is contained in the closure of at most two $Q \in \mathcal Q$. Therefore, by
using Lemma \ref{lma:shortcuts} and \eqref{eq:H1_of_A} we get
\begin{equation}\label{eq:manucubes}
 \sum_{Q\in\mathcal Q}\ell(Q)
\ge \frac{1}{10}\sum_{Q\in\mathcal Q}
\mathcal H^{1}(\gamma\cap \overline{Q})\geq \frac{1}{20} \mathcal H^1(A\cap B(w_i,2^{-k-3}))
\geq  2^{-k-9}.
\end{equation}
So, by \eqref{eq:W_Qbound}
\[
 \# \left\{Q' \in \mathcal D\,:\, Q' \subset A_k(x)\text{ and }\ell(Q') = \varepsilon 2^{-k}\right\} \ge \sum_{Q\in\mathcal Q}\frac{\ell(Q)}{\varepsilon 2^{-k+1}} \ge \varepsilon^{-1} 2^{-10}.
\]
Again, $\chi_k^\mathcal{D}(x) = 1$, concluding the proof.
\end{proof}

\section{Examples}\label{sec:examples}

In this section we show the sharpness of our estimate between the constant in the curve condition and the dimension of the boundary. We also show that the dimension estimate via the John condition is necessarily less sharp. Let us write the conclusions from the two sets of examples we consider in the following theorem.

 \begin{theorem}\label{thm:examples}
The following sets exist.
\begin{enumerate}
 \item For every $J \in (0,1/2)$ there exists a Jordan $J$-John domain $\Omega \subset \R^2$ for which
\[
 \dim_\mathcal{H}(\partial\Omega)  \ge 2 - \frac2{\log(2)} J.
\]
\item For every $p\in (1,2)$ and $C \in (72/(2-p),\infty)$ there exists a Jordan domain $\Omega \subset \R^2$ satisfying the curve condition \eqref{eq:curve_condition} with the constant $C$ and exponent $p$,  for which
\[
 \dim_\mathcal{H}(\partial\Omega)  \ge 2 - \frac{24}{\log(2)(2-p)C}.
\]
\item There exists a universal constant $c>0$ so that for every $p\in (1,2)$ and $C \in (c,\infty)$ there exists  a Jordan domain $\Omega \subset \R^2$ satisfying the curve condition \eqref{eq:curve_condition} with the constant $C$, but failing to be J-John domain for any 
\[
 J \ge c\left((2-p)C\right)^{\frac{1}{p-2}}.
\]
\end{enumerate}
\end{theorem}

Recall that the quasi-convexity of the complement of a domain $\Omega\subset \R^2$, and thus in particular the curve condition \eqref{eq:curve_condition}, implies that $\Omega$ is John, \cite{NV1991}. However, the curve condition \eqref{eq:curve_condition} does not imply that the complementary open set $\R^2 \setminus \overline{\Omega}$ would be even connected. In particular, the complementary domain does not have to be a John domain in the Jordan domain case.

In the rest of the section we prove the existence of the sets mentioned in Theorem \ref{thm:examples}.


\subsection{Cones}

The first set of examples shows the claim (3) in Theorem \ref{thm:examples}.  We consider a fixed square and on top of it attach a cone whose width is the parameter $\varepsilon$ that we vary in order to change the constants in the curve condition \eqref{eq:curve_condition} and the John condition.

\begin{example}\label{ex:1}
Let $ \varepsilon \in (0,1/2)$ and $1<p<2$. Let
\[\Omega \coloneqq \{(x^1,x^2) : |x^1|<1, |x^2+1|<1 \} \cup \{ (x^1,x^2) : |x^1| < (1-x^2) \varepsilon, x^2 \geq 0 \} \subset \mathbb{R}^2.\]
Then the following hold.
\begin{enumerate}
  \item[(i)]   For $z_1,z_2 \in \mathbb{R}^2 \setminus \Omega$ the curve condition \eqref{eq:curve_condition} holds with constant $C = \frac{c}{2-p}\varepsilon^{p-2}$ with some constant $c>0$ independent of $\varepsilon$.
  \item[(ii)] The set $\Omega$ fails to be $J$-John for any $J > \varepsilon$.
\end{enumerate}
\end{example}
\begin{proof}[Proof of (i)]

\begin{figure}
    \centering
    \psfrag{w1}{$w_1$}
    \psfrag{w2}{$w_2$}
    \psfrag{z1}{$z_1$}
    \psfrag{z2}{$z_2$}
    \psfrag{0}{$x$}
    \psfrag{gt}{$\gamma(t)$}
    \includegraphics[width=0.7\columnwidth]{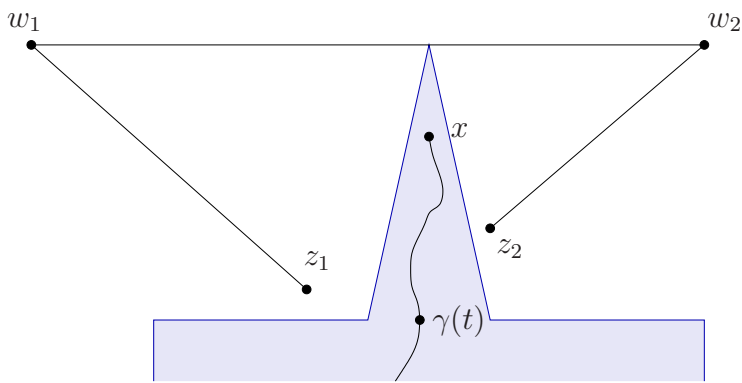}
    \caption{The failure of the John condition for $J > \varepsilon$ in Example \ref{ex:1} is seen by taking the point $x$ near the tip of the cone. Then every curve $\gamma$ connecting $x$ to a John center will fail the condition at a point $\gamma(t)$. The critical case for the curve condition \eqref{eq:curve_condition} is the case where $z_1$ and $z_2$ are on the opposite sides of the cone. Up to a constant, an optimal way to connect them goes through the points $w_1$ and $w_2$.}
    \label{fig:cone}
\end{figure}

Notice first that for $\Omega' \coloneqq \Omega \cup (0,1)\times (0,1)$
there exists a constant $C>0$ independent of $\varepsilon$ so that $\Omega'$ satisfies \eqref{eq:curve_condition} with this $C$.
Write $z_i = (z_i^1,z_i^2)$.
 Thus, we may assume that $-1 \le z_1^1 \le 0 \le z_2^1 \le 1$ and $0 \le z_1^2, z_2^2\le 1$.

 Let us define
 $w_1 = (z_1^1+z_1^2-1,1)$ and $w_2 = (z_2^1-z_2^2+1,1)$. We claim that the concatenation $\gamma$ of the line-segments $[z_1,w_1]$, $[w_1,w_2]$ and $[w_2,z_2]$ satisfies the curve condition with the claimed constant. See Figure \ref{fig:cone} for an illustration of the curve. For the lengths of the line-segments we have the estimates
 \[
  \|w_i-z_i\| = \sqrt{2}|z_i^2-1| \le \frac{\sqrt{2}}{\varepsilon}|z_i^1| \le \frac{\sqrt{2}}{\varepsilon}\|z_1-z_2\|
 \]
 and
 \begin{align*}
  \|w_1-w_2\| &  = |(z_1^1+z_1^2-1) - (z_2^1-z_2^2+1)| \\
  & \le |z_1^2-1| + |z_2^2-1| + |z_1^1-z_2^1| \le \frac{3}{\varepsilon}\|z_1-z_2\|. 
 \end{align*}
 Thus, we get
 \[
  \int_{[z_i,w_i]}\dist(z,\partial\Omega)^{1-p}\,\d s(z) \le 
  \int_0^{\frac{\sqrt{2}}{\varepsilon}\|z_1-z_2\|}\left(\frac{t}{\sqrt{2}}\right)^{1-p}\,\d t
  = \frac{2^{3/2-p}}{2-p}\varepsilon^{2-p}\|z_1-z_2\|^{2-p}
 \]
 and
 \[
  \int_{[w_1,w_2]}\dist(z,\partial\Omega)^{1-p}\,\d s(z) \le 
  2\int_0^{\frac{3}{\varepsilon}\|z_1-z_2\|}\left(\frac{t}{\sqrt{2}}\right)^{1-p}\,\d t
  = \frac{2^{(3-p)/2}3^{2-p}}{2-p}\varepsilon^{2-p}\|z_1-z_2\|^{2-p}.
 \]
 Combining the above estimates, the claim is proven.
\end{proof}

\begin{proof}[Proof of (ii)]
Figure \ref{fig:cone} shows the idea of the proof.
 Suppose $\Omega$ is a $J$-John domain with the John center $x_0 =(x_0^1,x_0^2) \in \Omega$. For $x_0^2 < x^2 <1$, consider a John curve $\gamma\colon [0,\ell(\gamma)] \to \Omega$ from $(0,x^2)$ to $(x_0^1,x_0^2)$. Let $t \in [0,\ell(\gamma)]$ be such that $\gamma(t) \in \R \times \{\max(0,x_0^2)\}$.
 Then,
 \[
  Jt \le \dist(\gamma(t),\partial\Omega) \le \varepsilon\min(1-x_0^2,1)
  \le \varepsilon \min\left(\frac{1-x_0^2}{x^2-x_0^2},\frac{1}{x^2-x_0^2}\right)t
  \le \varepsilon \frac{1-x_0^2}{x^2-x_0^2} t.
 \]
 Thus, by letting $x^2 \nearrow 1$, we see that $J \le \varepsilon$. 
\end{proof}

\subsection{Koch snowflakes}

The second set of examples showing the claims (1) and (2) in Theorem \ref{thm:examples} is the von Koch snowflake with varying contraction constant $\lambda$ as the parameter.

\begin{example}\label{ex:2}
Let us first recall the construction of the von Koch curve $K$ with parameter $\lambda \in [1/3,1/2)$.
It is defined as the attractor of iterated function system $\{ F_1, F_2,F_3,F_4 \}$, where $F_1,\ldots, F_4$ are the similitude mappings
\begin{align*}
F_1x &= Sx, \qquad
F_2 x   = T_{(\lambda,0)} R_\theta S x,\qquad 
F_3 x  = T_{(1/2,h)} R_{- \theta} S x, \qquad 
F_4 x = T_{(1-\lambda,0)} S x.
\end{align*}
Here $Sx = \lambda x$ is the scaling by $\lambda$, $R_{\tau}$
is the rotation of the plane by the angle $\tau$, the used rotation angle $\theta$ here is defined by $\cos \theta = (\frac{1}{2}-\lambda)/{\lambda}$, $T_a$ is the translation $T_a x = x+a$, and $h = \sqrt{\lambda - 1/4}$.
Recall that $K$ being the attractor means that it is the unique non-empty compact set satisfying
\[
 K = \bigcup_{i=1}^4 F_i(K).
\]

\begin{figure}
    \centering
    \psfrag{x0}{$x_0$}
    \psfrag{g}{$\gamma$}
    \includegraphics[width=0.6\columnwidth]{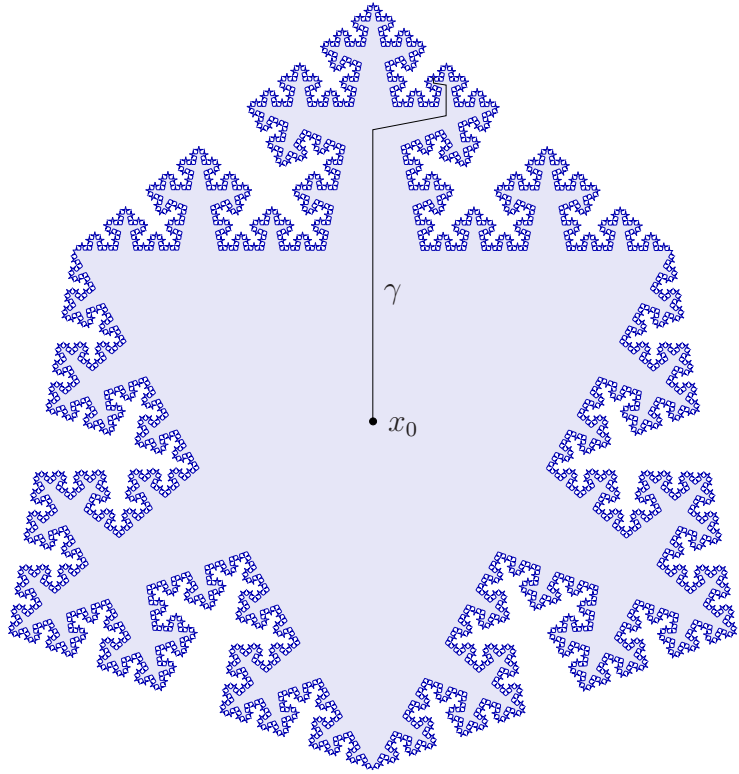}
    \caption{An illustration of the domain $\Omega$ bounded by three copies of a von Koch curve $K$ together with the John center $x_0$ and a John curve $\gamma$.}
    \label{fig:kochJohn}
\end{figure}

We define our domain $\Omega$ to be a snowflake domain whose boundary consists of three copies of $K$, see Figure \ref{fig:kochJohn}. More precisely, $\Omega$ is the bounded component of the set
\[
 \R^2 \setminus \bigcup_{i=1}^3\ G_i(K),
 \]
 where
 \[
  G_1x = x, \qquad G_2x = T_{(1,0)}R_{-\frac{2\pi}{3}}x, \qquad G_3x = T_{(1/2, - \sqrt{3}/2)} R_{\frac{2\pi}{3}} x. 
 \]
 
Since the iterated function system defining $K$ satisfies the open set condition, the Hausdorff dimension agrees with the similarity dimension, which gives
\begin{equation}\label{eq:Kochdim}
\dim_\mathcal{H}(\partial\Omega) = -\frac{\log(4)}{\log(\lambda)} \ge 2 - \frac{4}{\log(2)}\left(\frac12-\lambda\right).
\end{equation}
\medskip
We claim that the following hold.
\begin{enumerate}
 \item[(i)]The domain $\Omega$  is $\frac{\frac{1}{2}-\lambda}{\lambda}$-John, which is also optimal.
 \item[(ii)]The domain $\Omega$ satisfies the curve condition \eqref{eq:curve_condition} with $C = \frac{6\lambda^{2p-3}}{(2-p)(1/2 - \lambda)}$.
\end{enumerate}
\end{example}

Before proving the claims, let us introduce some additional notation for the Koch snowflake.
For $k\in \{0,1,\dots\}$, and a word $a_0a_1\ldots a_k \in \{1, 2, 3\} \times \{1, 2,3, 4\}^k$, we define the composed mapping
\[
F_{a_0\ldots a_k}  \coloneqq G_{a_0} \circ F_{a_1} \circ \cdots \circ F_{a_k} .
\]
Now, we set $K_{a_0\ldots a_k} \coloneqq F_{a_0\dots a_k} ( K)$.
Similarly, by defining $L \coloneqq [0,1]\times \{0\}$, we set
$L_{a_0 \ldots a_k } \coloneqq F_{a_0\ldots a_k} ( L).$
We also fix the following notation
$$\Delta_{a_0 \ldots a_k} = \textrm{ch}( L_{a_0 \ldots a_k 2} \cup L_{a_0 \ldots a_k 3})$$
$$T_{a_0 \ldots a_k } =  L_{a_0 \ldots a_k 2} \cap L_{a_0 \ldots a_k 3},$$
where $\textrm{ch}(A)$ denotes the convex hull of set $A$.

\begin{proof}[Proof of (i)]
Let us first show that $\Omega$ cannot be John with a constant better than $\frac{\frac{1}{2}- \lambda}{\lambda}$. The proof is similar to the proof of (ii) in Example \ref{ex:1}. Suppose that $\Omega$ is $J$-John with $x_0 \in \Omega$ the John center.
Let $k \in \mathbb{N}$ be such that 
$$
x_0 \not \in \textrm{ch}(L_{12a_1 \ldots a_k} \cup L_{13 b_1 \ldots b_k}) =:\Delta,
$$
where $a_j = 4$ and $b_j=1$ for all $1 \leq j \leq k$.
Notice that the triangle $\Delta$ is similar to $\Delta_1$ with both having the same top vertex $T=T_1$.
 Let $\gamma$ be a unit speed curve connecting $T$ to $x_0$ in $\Omega \cup \{T\}$.  
Let $x  \in \partial\Omega \cap (L_{12a_1\ldots a_k 4} \cup L_{13b_1 \ldots b_k1})$ and $t \in [0,\ell(\gamma)]$ be such that $\dist(\gamma(t),\partial \Omega) = \|\gamma(t)- x\| > 0$.
Then,
$$\frac{t}{ \dist(\gamma(t),\partial \Omega)} \geq \frac{\|T-\gamma(t)\|}{ \|\gamma(t) -  x\|}  \geq  \frac{\lambda}{ \frac{1}{2} - \lambda}.$$
Therefore, $J \leq \frac{ \frac{1}{2} - \lambda}{\lambda}$.
\medskip

Let us then show that $\Omega$ is $\frac{\frac{1}{2} - \lambda}{\lambda}$-John. Let $x_0$ be the barycenter of $\Omega$, and let $x_1 \in \Omega$ be the point connected to $x_0$ with $\gamma$. Figure \ref{fig:kochJohn} shows the idea behind the following construction of the John curve $\gamma$. In the case $x_1 \in \Delta _0\coloneqq \textrm{ch} (L_1 \cup L_2 \cup L_3)$ the claim is clear. 
Assume that $x_1 \in \Delta_{a_0 \ldots a_k}$, $k \geq 0$, $a_0 \in \{1,2,3\}$, $a_j \in \{1,2,3,4\}, 1 \leq j \leq k$. Let $P_{a_1 \ldots a_k} \in \Omega$ be the point on the line bisecting $\Delta_{a_0 \ldots a_k}$ through $T_{a_0 \ldots a_k}$, such that $\|T_{a_0 \ldots a_k} -  P_{a_1 \ldots a_k}\| = \frac{\lambda^{k+1}}{2h}$, where $h=\sqrt{\lambda - 1/4}$. 
Now the line segment $[x_1, P_{a_0 \ldots a_k} ]$ has length at most $\frac{\lambda^{k+1}}{ 2h}$ and \eqref{eq:John} holds for all $x \in [x_1, P_{a_1 \ldots a_k}]$ with $J=\frac{\frac{1}{2} - \lambda}{\lambda}$.

By symmetry and self-similarity, 
the points $P_{a_0} , P_{a_0 a_1}, \ldots , P_{a_0a_1\ldots a_k}$, where $a_0 \in \{1,2,3\}$ and $a_1, \ldots, a_k \in \{1,2,3,4\}$, have the following properties:

For $x = t P_{a_0 \dots a_m } + (1-t) P_{a_0 \ldots a_{m+1}}$, $t \in [0,1]$ and $m \geq 0$
$$\ell([P_{a_0 \ldots a_{m+1}} , x] ) =t \frac{\lambda^{m+1}(1-\lambda)}{2h}$$
and
$$\dist(\partial \Omega, P_{a_0 \ldots a_m}) \geq \frac{(\frac{1}{2}-\lambda) \lambda^{m}}{2h},
$$
which by the construction of $\Omega$ gives
\begin{align*}
\dist(\partial \Omega, x )& \geq (1-t) \frac{(\frac{1}{2}-\lambda) \lambda^{m+1}}{2h} + t \frac{(\frac{1}{2}-\lambda) \lambda^{m}}{2h}\\
& = [(1-t) \lambda +t ] \frac{(\frac{1}{2}- \lambda )\lambda ^{m}}{2h}.
\end{align*}

Therefore, for all $1 \leq m \leq k-1$ and $x \in [P_{a_0 \ldots a_{m+1}}, P_{a_0 \ldots a_m}]$
\begin{align*}
\ell(\gamma|_{x_1 \to x} )& = \ell([x_1, P_{a_0 \ldots a_k}]) + \sum _{j =m+2}^{k} \ell([P_{a_1 \ldots  a_{j-1}}, P_{a_1 \ldots a_j}]) + \ell([P_{a_1 \ldots a_{m+1}} ,x ])\\
& \leq \frac{\lambda^{k+1}}{2h } + \sum_{j=m+2}^k \frac{\lambda ^{j} (1- \lambda)}{2h} + t \frac{\lambda^{m+1}(1-\lambda)}{2h} \\
&= [\lambda (1-t) + t ] \frac{\lambda^{m+1}}{2h} \\
& \leq \frac{\lambda}{ \frac{1}{2} - \lambda} \dist(\partial \Omega , x ),
\end{align*}
where  $\gamma|_{x_1 \to x}$ denotes curve made of the line segments $$[x_1, P_{a_0 \ldots a_k }], [P_{a_0 \ldots a_k }, P_{a_1\ldots a_{k-1} }], \ldots, [P_{a_0 \ldots a_{m+2} }, P_{a_0 \ldots a_{m+1} }] , [P_{a_0 \ldots a_{m+1} } , x].$$
So \eqref{eq:John} holds for all $x \in  \gamma|_{x_1 \to P_{a_1}}$ with $J= \frac{\frac{1}{2}-\lambda}{\lambda}$ and \eqref{eq:John} still holds (with the same constant) when $\gamma|_{x_1\to P_{a_1}}$ is extended to $x_0$ with $[P_{a_1},x_0]$.
\end{proof}

\begin{proof}[Proof of (ii)]
 We will show that any two points $z_1,z_2 \in \R^2 \setminus \Omega$ can be connected by a curve $\gamma \subset \R^2 \setminus \Omega$ satisfying \eqref{eq:curve_condition} with $C = \frac{9\lambda^{3p-7}}{(2-p)(\frac12-\lambda)}$. First of all, we may assume without loss of generality that $z_1,z_2 \in \partial \Omega$. Secondly, we may assume that $z_1,z_2 \in K_1$. 
We now divide the proof into three cases, them being case 1: $z_1 \in K_{11}$, $z_2 \in K_{12}$, case 2: $z_1 \in K_{12}$, $z_2 \in K_{13}$, and case 3: $z_1 \in K_{11}$, $z_2 \in K_{13}\cup K_{14}$. Other cases follow then by symmetry, and from self-similarity by zooming in to the construction. We treat only the case 1 in detail, giving the ideas for the other two.
\medskip

\noindent
{\color{blue}\textsc{Case 1}:} $z_1 \in K_{11}$ and $z_2 \in K_{12}$.\\
Let us call $z_1',z_2'$ the orthogonal projections of $z_1$ and $z_2$ on the line-segment 
\[
I \coloneqq T_{(\lambda,0)}R_{(\pi-\theta)/2}L 
\]
 (a line-segment in mirroring $K_{11}$ and $K_{12}$).
We will define points $p_{z_1}$ and $p_{z_2}$ in $I$, that are connected to $z_1$ and $z_2$ by curves, which we will call $\gamma_1$ and $\gamma_2$.  We then join the points $p_{z_1}$ and $p_{z_2}$ with a line-segment. See Figure \ref{fig:kochcurve2} for an illustration.

\begin{figure}
    \centering
    \psfrag{p1}{$p_{z_1}$}
    \psfrag{p2}{$p_{z_2}$}
    \psfrag{z1}{${z_1}$}
    \psfrag{z2}{${z_2}$}
    \psfrag{I}{$I$}
    \includegraphics[width=0.8\columnwidth]{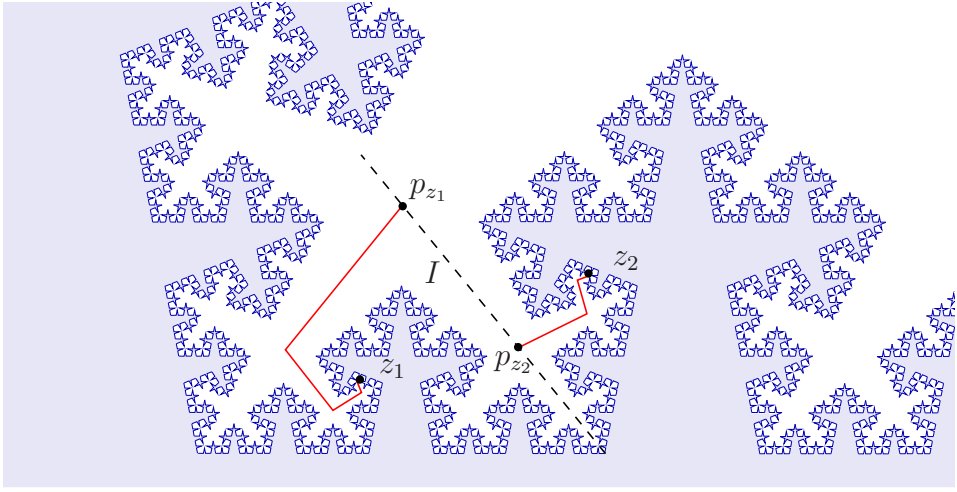}
    \caption{In the proof of the curve condition \eqref{eq:curve_condition}
    we consider three critical cases. Here, in the zoomed in picture of the case 1, the points $z_1$ and $z_2$ are connected to points $p_{z_1}$ and $p_{z_2}$ on a line-segment $I$ using the curves constructed in the proof of claim (i).}
    \label{fig:kochcurve2}
\end{figure}

Let us write $\{o\} \coloneqq  K_{11} \cap K_{12}$. If $z_1 = o$, we take $p_{z_1} = z_1$. If not, then there exists $ k \ge  1$ such that $z_1 \in K_{11a_1\ldots a_k}$ with $a_i = 4$ for all $ i < k$ and $a_k \ne 4$.
We can make a crude estimate
\begin{equation}\label{eq:kochestim1}
\|z_1 - z_1'\| \ge \left(\frac12-\lambda\right)\lambda^{k+2}.
\end{equation}
Now,
by the proof of the (i), $z_1$ can be connected to a point $p_{z_1} \in I$
by a John curve with John constant $\frac{1/2-\lambda}{\lambda}$ and length less than $\lambda^{k-1}$. Combining this with \eqref{eq:kochestim1}, we get
\begin{align}
\int_{\gamma_1} \dist(z,\partial \Omega)^{1-p}\,\d z &
\le 2 \int_0^{\lambda^{k-1}}\left(\frac{\frac12-\lambda}{\lambda}t\right)^{1-p}\,\d t \nonumber \\
& = \frac2{2-p}\left(\frac{\frac12-\lambda}{\lambda}\right)^{1-p}\left(\lambda^{k-1}\right)^{2-p} \label{eq:kochcase1_1} \\
& \le \frac2{2-p}\frac{\lambda^{3p-7}}{\frac12-\lambda}  \|z_1-z_1'\|^{2-p} \le \frac{C}{3}\|z_1-z_2\|^{2-p}.
\end{align}
By symmetry, with the same arguments we also find $p_{z_2}$ and the curve $\gamma_2$ connecting $z_2$ to $p_{z_2}$, and get
\begin{equation}\label{eq:kochcase1_2}
\int_{\gamma_2} \dist(z,\partial \Omega)^{1-p}\,\d z \le  \frac2{2-p}\frac{\lambda^{3p-7}}{\frac12-\lambda}  \|z_1-z_1'\|^{2-p} \le \frac{C}{3}\|z_1-z_2\|^{2-p}.
\end{equation}
For the line-segment $[p_{z_1},p_{z_2}]$, notice that we have 
\begin{align}
 \|p_{z_1}-p_{z_2}\| & \le \|z_1'-z_2'\| + \|p_{z_1}-z_1'\| + \|p_{z_2}-z_2'\|
 \le \|z_1'-z_2'\| + 2\lambda^{k-1}\\
 & \le \|z_1'-z_2'\| + \lambda^{-3}\left(\frac12-\lambda\right)^{-1}\left(\|z_1 - z_1'\| + \|z_2 - z_2'\|\right)\\
& \le 3\lambda^{-3}\left(\frac12-\lambda\right)^{-1}\|z_1 - z_2\|,
\end{align}
and thus
\begin{align}
\int_{[p_{z_1},p_{z_2}]} \dist(z,\partial \Omega)^{1-p}\,\d z & \le \int_0^{\|p_{z_1}-p_{z_2}\|} \left(\frac{\frac12-\lambda}{\lambda}t\right)^{1-p} \,\d t\\
& \le \left(\frac{\frac12-\lambda}{\lambda}\right)^{1-p}\frac1{2-p}\left(3\lambda^{-3}\left(\frac12-\lambda\right)^{-1}\|z_1 - z_2\|\right)^{2-p} \label{eq:kochcase1_3}  \\
& \le \frac{3^{2-p}}{2-p}\frac{\lambda^{3p-7}}{\frac12-\lambda}  \|z_1-z_2\|^{2-p} \\
& \le \frac{C}{3}\|z_1-z_2\|^{2-p}.
\end{align}
Combining \eqref{eq:kochcase1_1}, \eqref{eq:kochcase1_2}, and \eqref{eq:kochcase1_3}, we conclude the first case.
\medskip

\noindent
{\color{blue}\textsc{Case 2}:} $z_1 \in K_{12}$ and $z_2 \in K_{13}$.\\
In this case, we connect $z_1$ and $z_2$ to the unique point $p \in K_{12}\cap K_{13}$ by curves $\gamma_1$ and $\gamma_2$.
The estimate for $\gamma_1$ and $\gamma_2$ are exactly the same as in case 1. We connect $z_1$ to $p_{z_1}$ with a John curve and then $p_{z_1}$ to $p$ (instead of $z_1'$) with a line-segment.
\medskip

\noindent
{\color{blue}\textsc{Case 3}:} $z_1 \in K_{11}$ and $z_2 \in K_{13}\cup K_{14}$.\\
Similarly as in the second case, we can connect $z_1$ and $z_2$ to the unique point $p \in K_{12}\cap K_{13}$ obtaining the desired estimate also in this case.
\end{proof}


%
%
%
%
%
%

\end{document}